\documentclass[11pt,reqno]{amsart}

\usepackage[T1]{fontenc}
\usepackage[utf8]{inputenc}
\usepackage{lmodern}
\usepackage{microtype}
\usepackage{amsmath,amssymb,amsthm,mathtools}
\usepackage{xcolor}
\usepackage{hyperref}
\usepackage[nameinlink,capitalise,noabbrev]{cleveref}
\usepackage{geometry}
\hypersetup{
  colorlinks=true,
  linkcolor=blue!45!black,
  citecolor=blue!45!black,
  urlcolor=blue!45!black,
  pdftitle={A sharp isoperimetric inequality for the Neumann--Poincare operator in every dimension},
  pdfauthor={Matthew J. Colbrook and Siavash Sadeghi}
}

\newtheorem{theorem}{Theorem}[section]
\newtheorem{proposition}[theorem]{Proposition}
\newtheorem{lemma}[theorem]{Lemma}
\newtheorem{corollary}[theorem]{Corollary}
\theoremstyle{remark}
\newtheorem{remark}[theorem]{Remark}

\numberwithin{equation}{section}

\newcommand{\R}{\mathbb R}
\newcommand{\Sph}{\mathbb S}
\newcommand{\cH}{\mathcal H}
\newcommand{\cV}{\mathcal V}
\newcommand{\Tr}{\operatorname{Tr}}
\newcommand{\dd}{\,\mathrm d}
\newcommand{\ip}[2]{\left\langle #1,#2\right\rangle}
\newcommand{\duality}[2]{\left\langle #1,#2\right\rangle_{\partial\Omega}}
\newcommand{\norm}[1]{\left\lVert #1\right\rVert}
\newcommand{\restr}{\big|}

\title[A sharp inequality for the Neumann--Poincar\'e operator]
{A sharp isoperimetric inequality for the\\
Neumann--Poincar\'e operator in every dimension}

\author{Matthew J. Colbrook}
\address{Department of Applied Mathematics and Theoretical Physics, University of Cambridge, Wilberforce Road, Cambridge CB3 0WA, United Kingdom}
\email{m.colbrook@damtp.cam.ac.uk}

\author{Siavash Sadeghi}
\address{Department of Mathematics and Statistics, Mathematics Building, University of Reading, Whiteknights campus, Reading RG6 6AX, United Kingdom}
\email{s.sadeghi@pgr.reading.ac.uk}


\begin{document}

\begin{abstract}
Let $\Omega\subset\R^d$, $d\ge2$, be a bounded connected domain with boundary of class $C^{1,\alpha}$, where $0<\alpha<1$. For the adjoint Neumann--Poincar\'e operator $K^*_{\partial\Omega}$, normalised so that its distinguished eigenvalue is $1/2$, let $\lambda_j^+(\Omega)$ denote the upper min--max values on the mean-zero energy space. We prove
$$
   \sum_{j=1}^{d}\lambda_j^+(\Omega)\ge \frac{d-2}{2}.
$$
It follows that
$$
   \lambda_1^+(\Omega)\ge \frac{d-2}{2d},
$$
with equality if and only if $\Omega$ is a ball. In dimension three this proves the $1/6$-conjecture of Miyanishi and Suzuki. The proof uses the coordinate boundary-charge densities induced by uniform applied fields. Their energy Gram matrix is the perfect-conductor polarization tensor $M_\infty$. Positivity of a $2d\times2d$ Gram matrix yields the endpoint Hashin--Shtrikman inequality
$$
   |\Omega|\Tr(M_\infty^{-1})\le1,
$$
and bounds the trace of the compression of $K^*_{\partial\Omega}$ to the applied-field space. Equality in the inverse-trace inequality makes the interior Newtonian potential quadratic, so a converse to Newton's theorem identifies $\Omega$ as an ellipsoid. Equality in the spectral estimate also makes the Hessian of this potential isotropic, which forces the ellipsoid to be a ball.
\end{abstract}

\maketitle

\section{Introduction}

The Neumann--Poincar\'e operator is central to classical layer-potential theory. On a $C^{1,\alpha}$ boundary it is compact and, although generally not self-adjoint in $L^2$, Plemelj symmetrisation makes it self-adjoint in the single-layer energy space. Its nontrivial spectrum is consequently real and variational. Covariance of the kernel under Euclidean similarities makes this spectrum invariant under translations, rotations, and dilations. The operator appears in transmission and inverse problems and in the theory of polarization tensors \cite{AmmariKang2007}, as well as in quasistatic plasmonic resonance \cite{Grieser2014}. Its energy-space spectral theory is developed in \cite{KhavinsonPutinarShapiro2007,PerfektPutinar2014}.

We denote by $\lambda_j^+(\Omega)$ the upper min--max values of $K^*_{\partial\Omega}$ on the mean-zero energy space. Equivalently, these are the positive eigenvalues in nonincreasing order, followed by zeros. The variational framework underlying the spectral isoperimetric question goes back to Poincar\'e's 1897 theory \cite{Poincare1897,KhavinsonPutinarShapiro2007}. In three dimensions, the spectrum of a sphere below the distinguished eigenvalue $1/2$ begins with the degree-one eigenvalue $1/6$, of multiplicity three. Miyanishi and Suzuki conjectured, in an equivalent normalisation, that every smooth domain $\Omega\subset\R^3$ whose boundary is a simply connected closed surface satisfies
$$
   \lambda_1^+(\Omega)\ge\frac16,
$$
with equality only for a ball \cite[Conjecture~1]{MiyanishiSuzuki2017}. Using Martensen's cluster-sum identity \cite{Martensen1999}, Miyanishi and Suzuki proved the conjecture within the class of ellipsoids \cite[Theorem~4.3]{MiyanishiSuzuki2017}. Ando, Kang, Miyanishi, and Ushikoshi subsequently referred to this as the ``$1/6$-conjecture'' and, using Grieser's Hadamard-type variation formula, proved that the sum of the first variations in each spherical eigenvalue cluster vanishes \cite{Grieser2014,AndoKangMiyanishiUshikoshi2019}. The global comparison over the conjectured class remained unresolved for more than a decade; see also \cite{DallaRivaLambertiLuzziniMusolino2025}.

In $\R^d$, the degree-one spherical harmonics have eigenvalue
$$
   c_d:=\frac{d-2}{2d}.
$$
This suggests the same extremal problem in every dimension. The following theorem proves the resulting bound and, more strongly, controls the sum of the first $d$ upper min--max values. For $d=3$, \eqref{eq:main-max} is the $1/6$-conjecture. For $d=2$, the lower bound is zero, while the equality statement remains geometric: the upper spectral edge on the mean-zero energy space vanishes only for a disc.

\begin{theorem}[Main theorem]\label{thm:main}
Let $d\ge2$, and let $\Omega\subset\R^d$ be a bounded connected domain with boundary of class $C^{1,\alpha}$ for some $0<\alpha<1$. Then
\begin{equation}\label{eq:main-sum}
   \sum_{j=1}^{d}\lambda_j^+(\Omega)\ge\frac{d-2}{2}.
\end{equation}
In particular,
\begin{equation}\label{eq:main-max}
   \lambda_1^+(\Omega)\ge\frac{d-2}{2d}.
\end{equation}
Equality holds in \eqref{eq:main-max} if and only if $\Omega$ is a ball.
\end{theorem}

We outline the proof. Let $\nu$ denote the outward unit normal, $S$ the single-layer operator, $\ip{\cdot}{\cdot}_*$ the single-layer energy inner product, and $P_\Omega$ the Newtonian potential of $\Omega$. On the mean-zero energy space, set
$$
   A:=\frac12I-K^*_{\partial\Omega},
$$
which is invertible by \cref{lem:A-positive}. For $a\in\R^d$, define
$$
   n_a:=a\cdot\nu,
   \qquad
   \psi_a:=A^{-1}n_a.
$$
The boundary-charge density $\psi_a$ is induced by the uniform applied field $a$. Their span is the $d$-dimensional space
$$
   \cV:=\{\psi_a:a\in\R^d\}.
$$
If $M_\infty$ denotes the perfect-conductor polarization tensor, then the energy form and the $A$-form on $\cV$ are
\begin{equation}\label{eq:intro-two-forms}
   \ip{\psi_a}{\psi_b}_*=a^{\mathsf T}M_\infty b,
   \qquad
   \ip{\psi_a}{A\psi_b}_*=|\Omega|\,a\cdot b.
\end{equation}
If $\mu_1\ge\cdots\ge\mu_d>0$ are the eigenvalues of $M_\infty$, the corresponding Ritz values of $K^*_{\partial\Omega}$ are
\begin{equation}\label{eq:intro-ritz}
   r_j=\frac12-\frac{|\Omega|}{\mu_j}.
\end{equation}

The second ingredient is the endpoint inverse-trace inequality
\begin{equation}\label{eq:intro-hs}
   |\Omega|\Tr(M_\infty^{-1})\le1.
\end{equation}
It is the perfect-conductor endpoint of the Hashin--Shtrikman polarization bound. For the standard basis $e_1,\dots,e_d$, let $n_i=n_{e_i}$, $\psi_i=\psi_{e_i}$, and $N_{ij}=\ip{n_i}{n_j}_*$. The Gram matrix of $\psi_1,\dots,\psi_d,n_1,\dots,n_d$ satisfies
\begin{equation}\label{eq:intro-block}
   \begin{pmatrix}
      M_\infty&|\Omega|I_d\\
      |\Omega|I_d&N
   \end{pmatrix}\succeq0.
\end{equation}
Here $\succeq$ denotes the Loewner order. The identities $\partial_jP_\Omega=Sn_j$ and $\Delta P_\Omega=-1$ in $\Omega$ give $\Tr(N)=|\Omega|$. Taking the Schur complement in \eqref{eq:intro-block} and then the trace proves \eqref{eq:intro-hs}. It follows from \eqref{eq:intro-ritz} that
$$
   \sum_{j=1}^dr_j
   =\frac d2-|\Omega|\Tr(M_\infty^{-1})
   \ge\frac{d-2}{2}.
$$
Ky Fan's variational principle applied to the compression on $\cV$ now gives \eqref{eq:main-sum}.

The proof has two rigidity stages. Equality in \eqref{eq:intro-hs} is equivalent to $n_1,\dots,n_d\in\cV$, which makes every component of $\nabla P_\Omega$ affine in $\Omega$. An all-dimensional converse to Newton's theorem then shows that $\Omega$ is an ellipsoid. Earlier forms of this converse are due to Dive and Nikliborc in dimension three, to H\"older in dimension two, and to DiBenedetto--Friedman in arbitrary dimension; we use Karp's formulation \cite{Dive1931,Nikliborc1932,Holder1932,DiBenedettoFriedman1986,Karp1994}. Equality in \eqref{eq:main-max} further gives $M_\infty=d|\Omega|I_d$ and
$$
   P_\Omega(x)=-\frac{|x|^2}{2d}+b\cdot x+C
   \qquad (x\in\Omega).
$$
The Hessian is isotropic, and the ellipsoid is therefore a ball.

\section{Layer potentials and the energy realisation}\label{sec:energy}

Throughout, $d\ge2$ and $\Omega\subset\R^d$ satisfies the hypotheses of \cref{thm:main}. All Sobolev spaces and duality pairings are real. Let $\nu$ denote the outward unit normal, let $|\Omega|$ be the $d$-dimensional volume of $\Omega$, and let $|\Sph^{d-1}|$ be the surface measure of the unit sphere. We use the fundamental solution of $\Delta$ given by
\begin{equation}\label{eq:fundamental}
   \Gamma_d(x)=
   \begin{cases}
      \displaystyle \frac1{2\pi}\log|x|, & d=2,\\[1ex]
      \displaystyle -\frac1{(d-2)|\Sph^{d-1}|}|x|^{2-d}, & d\ge3,
   \end{cases}
   \qquad \Delta\Gamma_d=\delta_0.
\end{equation}
For smooth densities, the single-layer potential is
\begin{equation}\label{eq:single-layer}
   S\varphi(x)=\int_{\partial\Omega}\Gamma_d(x-y)\varphi(y)\,\dd\sigma(y).
\end{equation}
The Neumann--Poincar\'e operator and its adjoint are initially defined on smooth boundary data by
\begin{equation}\label{eq:np-def}
   Kf(x)=\operatorname{p.v.}\int_{\partial\Omega}
      \partial_{\nu_y}\Gamma_d(x-y)f(y)\,\dd\sigma(y),\qquad
   K^*\varphi(x)=\operatorname{p.v.}\int_{\partial\Omega}
      \partial_{\nu_x}\Gamma_d(x-y)\varphi(y)\,\dd\sigma(y).
\end{equation}
They extend to bounded operators on $H^{1/2}(\partial\Omega)$ and $H^{-1/2}(\partial\Omega)$, respectively, while the boundary trace of $S$ maps $H^{-1/2}(\partial\Omega)$ to $H^{1/2}(\partial\Omega)$. We use $S$ for both the potential and its trace. The notation $\duality{\varphi}{f}$ denotes the duality between $H^{-1/2}(\partial\Omega)$ and $H^{1/2}(\partial\Omega)$, with the distribution in the first slot. These mapping properties and the jump relations are standard; see, for example, \cite{McLean2000}. The interior normal derivative satisfies
\begin{equation}\label{eq:jump}
   \partial_\nu S\varphi\restr_-=
      \left(-\frac12I+K^*\right)\varphi.
\end{equation}

Define
\begin{equation}\label{eq:H0}
   \cH_0^*:=\left\{\varphi\in H^{-1/2}(\partial\Omega):
       \duality{\varphi}{1}=0\right\}.
\end{equation}
In two dimensions the zero-mean restriction removes the logarithmic ambiguity and makes the energy coercive. In dimensions $d\ge3$ it removes the distinguished mode. Since $K1=\frac12$, the space \eqref{eq:H0} is invariant under $K^*$.

On $\cH_0^*$ set
\begin{equation}\label{eq:energy-ip}
   \ip{\varphi}{\psi}_*:=-\duality{\varphi}{S\psi}.
\end{equation}
This is a Hilbert inner product equivalent to the $H^{-1/2}$ norm on $\cH_0^*$. Plemelj's identity
\begin{equation}\label{eq:plemelj}
   SK^*=KS
\end{equation}
shows that $K^*$ is self-adjoint with respect to \eqref{eq:energy-ip}. Since $\partial\Omega$ is $C^{1,\alpha}$, the operator $K^*$ is compact on this space. We henceforth suppress $\partial\Omega$ from $S$, $K$, and $K^*$ unless the underlying boundary changes.

Set
\begin{equation}\label{eq:A-def}
   A:=\frac12I-K^*.
\end{equation}
The following lemma provides the positivity and invertibility needed below.

\begin{lemma}\label[lemma]{lem:A-positive}
The operator $A$ is positive and boundedly invertible on $\cH_0^*$. More precisely,
\begin{equation}\label{eq:A-interior-energy}
   \ip{\varphi}{A\varphi}_*
   =\int_\Omega|\nabla S\varphi|^2\,\dd x
   \qquad(\varphi\in\cH_0^*).
\end{equation}
\end{lemma}

\begin{proof}
By self-adjointness of $A$ in the energy inner product and by \eqref{eq:jump},
$$
   \ip{\varphi}{A\varphi}_*
   =-\duality{A\varphi}{S\varphi}
   =\duality{\partial_\nu S\varphi\restr_-}{S\varphi}.
$$
The weak Green identity gives \eqref{eq:A-interior-energy}. If the right-hand side vanishes, then $S\varphi$ is constant in the connected domain $\Omega$. Since $\varphi$ has mean zero,
$$
   \norm{\varphi}_*^2=-\duality{\varphi}{S\varphi}=0,
$$
so $\varphi=0$. Thus $A$ is injective. It is a compact perturbation of $\frac12I$ and hence Fredholm of index zero, so it is invertible.
\end{proof}

The mean-zero restriction retains every eigenspace apart from the distinguished mode. The equilibrium density gives a $1/2$-eigenfunction of $K^*$ \cite{KhavinsonPutinarShapiro2007}. This eigenvalue is simple: two linearly independent eigenfunctions would have a nonzero linear combination in $\cH_0^*$, contradicting the injectivity of $A$. Moreover, if $K^*\varphi=\lambda\varphi$ with $\lambda\ne1/2$, then
$$
   \lambda\duality{\varphi}{1}
   =\duality{K^*\varphi}{1}
   =\duality{\varphi}{K1}
   =\frac12\duality{\varphi}{1},
$$
and hence $\varphi\in\cH_0^*$. Thus restriction to $\cH_0^*$ removes precisely the distinguished eigenvalue.

For a compact self-adjoint operator $T$ on $\cH_0^*$, define
\begin{equation}\label{eq:upper-variational}
   \lambda_j^+(T):=
   \inf_{\substack{F\subset\cH_0^*\\\dim F=j-1}}
   \ \sup_{\substack{0\ne u\in F^\perp}}
   \frac{\ip{u}{Tu}_*}{\norm{u}_*^2}.
\end{equation}
These min--max values consist of the positive eigenvalues in nonincreasing order, repeated according to multiplicity, followed by zeros; a zero in this sequence need not be an eigenvalue. We use the abbreviation $\lambda_j^+(\Omega)=\lambda_j^+(K^*_{\partial\Omega})$. In particular,
\begin{equation}\label{eq:lambda-edge}
   \lambda_1^+(\Omega)=\sup\sigma(K^*\restr_{\cH_0^*}).
\end{equation}

We calibrate the normalisation on a ball.

\begin{proposition}[Sphere calibration]\label[proposition]{prop:sphere}
Let $B$ be a ball in $\R^d$. On the spherical harmonics of degree $\ell\ge1$,
\begin{equation}\label{eq:sphere-spectrum}
   K^*_{\partial B}Y_\ell
   =\frac{d-2}{2(2\ell+d-2)}Y_\ell.
\end{equation}
Consequently,
\begin{equation}\label{eq:sphere-top}
   \lambda_1^+(B)=\cdots=\lambda_d^+(B)=\frac{d-2}{2d},
\end{equation}
and
\begin{equation}\label{eq:sphere-sum}
   \sum_{j=1}^d\lambda_j^+(B)=\frac{d-2}{2}.
\end{equation}
For $d=2$, all nonconstant eigenvalues on a circle are zero.
\end{proposition}

\begin{proof}
By scaling it is enough to take the unit ball. If $Y_\ell$ is a spherical harmonic of degree $\ell\ge1$, separation of variables gives
$$
   S_{\partial B} Y_\ell(r\theta)=-\frac{r^\ell}{2\ell+d-2}Y_\ell(\theta),
   \qquad 0\le r<1.
$$
Taking the interior normal derivative at $r=1$ and comparing with \eqref{eq:jump} yields \eqref{eq:sphere-spectrum}. The degree-one space has dimension $d$, and the displayed eigenvalues decrease with $\ell$ when $d\ge3$. The remaining assertions follow.
\end{proof}

\section{Applied fields and the polarization tensor}\label{sec:applied}

For $a\in\R^d$, let
\begin{equation}\label{eq:na-psia}
   n_a:=a\cdot\nu,
   \qquad
   \psi_a:=A^{-1}n_a.
\end{equation}
The divergence theorem gives $\duality{n_a}{1}=0$, so these quantities belong to $\cH_0^*$. For the standard coordinate vectors, let $n_i=n_{e_i}$ and $\psi_i=\psi_{e_i}$.

\begin{lemma}[Interior affine field]\label[lemma]{lem:affine-field}
For every $a\in\R^d$ there is a constant $c_a$ such that
\begin{equation}\label{eq:Spsi-affine}
   S\psi_a(x)=-a\cdot x+c_a,
   \qquad x\in\Omega.
\end{equation}
\end{lemma}

\begin{proof}
Since $A\psi_a=n_a$, the jump relation gives
$$
   \partial_\nu S\psi_a\restr_-=-n_a=-\partial_\nu(a\cdot x).
$$
Thus the harmonic function $S\psi_a+a\cdot x$ has zero Neumann data on the connected domain $\Omega$ and is constant.
\end{proof}

Define the perfect-conductor polarization tensor $M=M_\infty(\Omega)\in\R^{d\times d}$ by
\begin{equation}\label{eq:M-def}
   M_{ij}:=\duality{\psi_i}{x_j}
   =\int_{\partial\Omega}x_j\psi_i\,\dd\sigma.
\end{equation}
This is the usual infinite-contrast polarization tensor in the normalisation corresponding to \eqref{eq:fundamental}; see \cite[Chapter~4]{AmmariKang2007} and \cite{KangMilton2008}. The identities below connect it to the spectral problem.

\begin{proposition}[Two applied-field forms]\label[proposition]{prop:two-forms}
For all $a,b\in\R^d$,
\begin{align}
   \ip{\psi_a}{\psi_b}_*&=a^{\mathsf T}Mb,\label{eq:M-gram}\\
   \ip{\psi_a}{A\psi_b}_*&=|\Omega|\,a\cdot b.\label{eq:A-form}
\end{align}
In particular, $M$ is symmetric positive definite and
\begin{equation}\label{eq:V-def}
   \cV:=\{\psi_a:a\in\R^d\}
\end{equation}
has dimension $d$.
\end{proposition}

\begin{proof}
Using \cref{lem:affine-field} and the fact that $\psi_a$ has zero mean,
$$
   \ip{\psi_a}{\psi_b}_*
   =-\duality{\psi_a}{S\psi_b}
   =\duality{\psi_a}{b\cdot x}
   =a^{\mathsf T}Mb.
$$
This proves \eqref{eq:M-gram}, including symmetry. Since $A\psi_b=n_b$ and $A$ is self-adjoint in the energy product,
\begin{align*}
   \ip{\psi_a}{A\psi_b}_*
   =\ip{\psi_a}{n_b}_*
     =\ip{n_b}{\psi_a}_*
     =-\duality{n_b}{S\psi_a}
     =\int_{\partial\Omega}(a\cdot x)(b\cdot\nu)\,\dd\sigma
     =|\Omega|\,a\cdot b,
\end{align*}
where the last equality is the divergence theorem. If $\psi_a=0$, then \eqref{eq:A-form} gives $|\Omega|\,|a|^2=0$, and hence $a=0$. Thus $a\mapsto\psi_a$ is injective. Equation \eqref{eq:M-gram} now shows that $M$ is positive definite.
\end{proof}

Let $\mu_1\ge\cdots\ge\mu_d>0$ be the eigenvalues of $M$, counted with multiplicity. Combining \eqref{eq:M-gram} and \eqref{eq:A-form} with $K^*=\frac12I-A$ reduces the Ritz problem on $\cV$ to the generalised matrix eigenvalue problem for the pair $\bigl(\frac12M-|\Omega|I_d,M\bigr)$.

\begin{corollary}[Applied-field Ritz values]\label[corollary]{cor:ritz}
The Ritz values of $K^*$ on $\cV$, with respect to the energy inner product, are
\begin{equation}\label{eq:ritz-values}
   r_j=\frac12-\frac{|\Omega|}{\mu_j},
   \qquad j=1,\dots,d.
\end{equation}
In particular,
\begin{equation}\label{eq:ritz-trace}
   \sum_{j=1}^dr_j=\frac d2-|\Omega|\Tr(M^{-1}).
\end{equation}
\end{corollary}

\begin{proof}
In the coordinates $a\mapsto\psi_a$, the Gram matrix of the energy inner product is $M$, whereas the form of $K^*$ is
$$
   \frac12M-|\Omega|I_d.
$$
Diagonalising the symmetric positive matrix $M$ gives \eqref{eq:ritz-values} and \eqref{eq:ritz-trace}.
\end{proof}

\section{An endpoint Hashin--Shtrikman inequality in every dimension}\label{sec:HS}

The classical finite-contrast polarization tensor satisfies a family of Hashin--Shtrikman bounds. We derive the inverse-trace endpoint needed here directly from a block Gram matrix.

Define the normal Gram matrix $N\in\R^{d\times d}$ by
\begin{equation}\label{eq:N-def}
   N_{ij}:=\ip{n_i}{n_j}_*.
\end{equation}

\begin{lemma}[Normal trace identity]\label[lemma]{lem:N-trace}
The matrix $N$ is symmetric positive definite and
\begin{equation}\label{eq:N-trace}
   \Tr(N)=|\Omega|.
\end{equation}
\end{lemma}

\begin{proof}
The energy inner product makes $N$ symmetric and nonnegative. If $a^{\mathsf T}Na=0$, then $n_a=0$ in $\cH_0^*$. Hence $a\cdot\nu=0$ on $\partial\Omega$, and the divergence theorem applied to $(a\cdot x)a$ gives $|\Omega|\,|a|^2=0$; thus $a=0$.

Consider the Newtonian potential, with the logarithmic interpretation when $d=2$,
\begin{equation}\label{eq:P-def}
   P_\Omega(x):=-\int_\Omega\Gamma_d(x-y)\,\dd y.
\end{equation}
In the sense of distributions on $\R^d$,
\begin{equation}\label{eq:gradP}
   \partial_jP_\Omega(x)
   =\int_{\partial\Omega}\Gamma_d(x-y)\nu_j(y)\,\dd\sigma(y)
   =S n_j(x).
\end{equation}
Indeed, this follows by differentiating the characteristic function of $\Omega$. Since $\chi_\Omega\in L^p(\R^d)$ for every finite $p$, local elliptic regularity gives $P_\Omega\in W_{\mathrm{loc}}^{2,p}(\R^d)$. Choosing $p>d$ shows that $P_\Omega\in C^1(\R^d)$, so \eqref{eq:gradP} holds pointwise on $\R^d$ and the boundary traces of its two sides agree. The divergence theorem and \eqref{eq:energy-ip} now give
\begin{equation}\label{eq:N-Hessian}
   N_{ij}
   =-\duality{n_i}{S n_j}
   =-\int_\Omega\partial_i\partial_jP_\Omega\,\dd x.
\end{equation}
Since $\Delta P_\Omega=-1$ in $\Omega$, taking the trace in \eqref{eq:N-Hessian} yields \eqref{eq:N-trace}.
\end{proof}

\begin{theorem}[Endpoint inverse-trace inequality]\label{thm:endpoint-HS}
Under the hypotheses of \cref{thm:main},
\begin{equation}\label{eq:endpoint-HS}
   |\Omega|\Tr(M^{-1})\le1.
\end{equation}
Equality holds if and only if $\Omega$ is an ellipsoid.
\end{theorem}

\begin{proof}
By \cref{prop:two-forms}, the Gram matrix of the $2d$ vectors
$$
   \psi_1,\dots,\psi_d,n_1,\dots,n_d
$$
in the energy space is
\begin{equation}\label{eq:block-Gram}
   G:=
   \begin{pmatrix}
      M&|\Omega|I_d\\
      |\Omega|I_d&N
   \end{pmatrix}\succeq0.
\end{equation}
Since $M$ is positive definite, its Schur complement is nonnegative:
\begin{equation}\label{eq:Schur}
   N-|\Omega|^2M^{-1}\succeq0.
\end{equation}
Taking traces and using \eqref{eq:N-trace} gives
$$
   |\Omega|=\Tr(N)\ge |\Omega|^2\Tr(M^{-1}),
$$
which is \eqref{eq:endpoint-HS}.

Equality can be characterised through the Schur complement. Since
$N-|\Omega|^2M^{-1}$ is positive semidefinite, equality in the trace
inequality is equivalent to
\begin{equation}\label{eq:Schur-zero}
N-|\Omega|^2M^{-1}=0.
\end{equation}
This matrix is the Gram matrix of the residuals
\begin{equation}\label{eq:residuals}
   q_j:=n_j-|\Omega|\sum_{i=1}^d(M^{-1})_{ij}\psi_i,
   \qquad j=1,\dots,d.
\end{equation}
Indeed, expanding $\ip{q_i}{q_j}_*$ and using the blocks in \eqref{eq:block-Gram} gives $N-|\Omega|^2M^{-1}$. Thus equality in \eqref{eq:endpoint-HS} is equivalent to $q_j=0$ for every $j$, or equivalently
\begin{equation}\label{eq:n-in-V}
   n_1,\dots,n_d\in\cV.
\end{equation}
If \eqref{eq:n-in-V} holds, then for each $j$ there is $\beta_j\in\R^d$ such that $n_j=\psi_{\beta_j}$. By \cref{lem:affine-field},
$$
   \partial_jP_\Omega=S n_j=-\beta_j\cdot x+c_j
   \qquad\text{in }\Omega.
$$
Hence every first derivative of $P_\Omega$ is affine, and $P_\Omega$ agrees in $\Omega$ with a quadratic polynomial $Q$ satisfying $\Delta Q=-1$. We first show that $\R^d\setminus\overline\Omega$ is connected. If $D$ were a bounded component of this set, let $\nu_D$ denote its outward unit normal. The function $P_\Omega$ is harmonic in $D$, and its $C^1$ regularity gives $\nabla P_\Omega=\nabla Q$ on $\partial D$. The divergence theorem on $D$ would therefore give
$$
   0=\int_{\partial D}\partial_{\nu_D}P_\Omega\,\dd\sigma
   =\int_{\partial D}\partial_{\nu_D}Q\,\dd\sigma
   =\int_D\Delta Q\,\dd x
   =-|D|,
$$
a contradiction. Thus $\R^d\setminus\overline\Omega$ is connected. Karp's all-dimensional form of the converse to Newton's theorem now shows that $\Omega$ is an ellipsoid \cite{Karp1994}. 

Conversely, if $\Omega$ is an ellipsoid, Newton's theorem says that $P_\Omega$ is quadratic in $\Omega$. Thus for each $j$ there are $\beta_j\in\R^d$ and $c_j\in\R$ such that
$$
   S n_j=\partial_jP_\Omega=-\beta_j\cdot x+c_j.
$$
Taking the interior normal derivative and using \eqref{eq:jump} gives
$$
   -A n_j=-n_{\beta_j},
   \qquad\text{hence}\qquad
   n_j=A^{-1}n_{\beta_j}=\psi_{\beta_j}\in\cV.
$$
Therefore all residuals \eqref{eq:residuals} vanish and equality holds.
\end{proof}

\begin{remark}[Finite contrast]\label{rem:finite-contrast}
For $k>1$, let
$$
   M(k)_{ij}:=\int_{\partial\Omega}x_j
   \left(\frac{k+1}{2(k-1)}I-K^*\right)^{-1}n_i\,\dd\sigma.
$$
The classical inverse-trace Hashin--Shtrikman inequality \cite[Theorem~4.16]{AmmariKang2007} states that
$$
   |\Omega|\Tr(M(k)^{-1})\le\frac{k+d-1}{k-1}.
$$
The right-hand side tends to $1$ as $k\to\infty$, and the resolvent converges on the mean-zero subspace to $A^{-1}$. Thus \eqref{eq:endpoint-HS} is its perfect-conductor endpoint. Further polarization-tensor bounds appear in \cite{Lipton1993,CapdeboscqVogelius2006,KangMilton2008}.
\end{remark}

\section{The spectral inequalities}\label{sec:spectral}

We combine the endpoint bound with the variational principle. Let $\Pi_\cV$ denote the energy-orthogonal projection onto the trial space $\cV$.

\begin{theorem}[Ky Fan bound]\label{thm:kyfan}
One has
\begin{equation}\label{eq:kyfan-bound}
   \sum_{j=1}^d\lambda_j^+(\Omega)
   \ge \Tr\bigl(\Pi_\cV K^*\restr_{\cV}\bigr)
   =\frac d2-|\Omega|\Tr(M^{-1})
   \ge\frac{d-2}{2}.
\end{equation}
\end{theorem}

\begin{proof}
For each finite-dimensional subspace $W\subset\cH_0^*$, let $\Pi_W$ be the energy-orthogonal projection onto $W$. Ky Fan's maximum principle gives
\begin{equation}\label{eq:kyfan-principle}
   \sum_{j=1}^d\lambda_j^+(\Omega)
   =\sup_{\substack{W\subset\cH_0^*\\\dim W=d}}
       \Tr\bigl(\Pi_WK^*\restr_W\bigr).
\end{equation}
Taking $W=\cV$, applying \cref{cor:ritz}, and using \cref{thm:endpoint-HS} proves \eqref{eq:kyfan-bound}.
\end{proof}

\begin{corollary}\label[corollary]{cor:max-bound}
For every $d\ge2$,
\begin{equation}\label{eq:max-bound}
   \lambda_1^+(\Omega)\ge\frac{d-2}{2d}.
\end{equation}
\end{corollary}

\begin{proof}
Since $\lambda_1^+\ge\lambda_j^+\ge0$,
$$
   d\lambda_1^+(\Omega)
   \ge\sum_{j=1}^d\lambda_j^+(\Omega)
   \ge\frac{d-2}{2}.
$$
\end{proof}

\begin{corollary}[The $1/6$-conjecture]\label[corollary]{cor:one-sixth}
If $\Omega\subset\R^3$ is a bounded connected domain with boundary of class $C^{1,\alpha}$, where $0<\alpha<1$, then
$$
   \sup\sigma\bigl(K^*\restr_{\cH_0^*}\bigr)\ge\frac16.
$$
Equivalently, the largest Neumann--Poincar\'e eigenvalue distinct from $1/2$ is at least $1/6$.
\end{corollary}

\section{Rigidity at the sharp constant}\label{sec:rigidity}

We classify equality in \eqref{eq:max-bound}. Equality will force the quadratic Newtonian potential to have isotropic Hessian.

\begin{theorem}[Rigidity]\label{thm:rigidity}
Equality holds in \eqref{eq:max-bound} if and only if $\Omega$ is a ball.
\end{theorem}

\begin{proof}
A ball attains equality by \cref{prop:sphere}. Conversely, suppose
\begin{equation}\label{eq:equality-assume}
   \lambda_1^+(\Omega)=c_d=\frac{d-2}{2d}.
\end{equation}
Every Ritz value $r_j$ in \eqref{eq:ritz-values} is bounded above by the spectral supremum $c_d$. On the other hand,
$$
   \sum_{j=1}^dr_j\ge\frac{d-2}{2}=dc_d.
$$
It follows that
\begin{equation}\label{eq:all-ritz-equal}
   r_1=\cdots=r_d=c_d.
\end{equation}
By \eqref{eq:ritz-values}, all eigenvalues of $M$ are $d|\Omega|$, and hence
\begin{equation}\label{eq:M-isotropic}
   M=d|\Omega|I_d.
\end{equation}
In particular, equality holds in \eqref{eq:endpoint-HS}. From \eqref{eq:Schur}, \eqref{eq:N-trace}, and \eqref{eq:M-isotropic},
$$
   N\succeq\frac{|\Omega|}{d}I_d,
   \qquad
   \Tr(N)=|\Omega|,
$$
so
\begin{equation}\label{eq:N-isotropic}
   N=\frac{|\Omega|}{d}I_d.
\end{equation}
For $a\in\R^d$, equations \eqref{eq:M-gram}, \eqref{eq:A-form}, \eqref{eq:M-isotropic}, and \eqref{eq:N-isotropic} give
\begin{align*}
   \norm{\psi_a-d\,n_a}_*^2
   =\norm{\psi_a}_*^2-2d\ip{\psi_a}{n_a}_*+d^2\norm{n_a}_*^2
   =d|\Omega|\,|a|^2-2d|\Omega|\,|a|^2+d|\Omega|\,|a|^2=0.
\end{align*}
Thus
\begin{equation}\label{eq:psi-dn}
   \psi_a=d\,n_a
   \qquad(a\in\R^d).
\end{equation}
Applying \cref{lem:affine-field} and dividing by $d$ yields
\begin{equation}\label{eq:Sn-isotropic}
   S n_a(x)=-\frac1d a\cdot x+\widetilde c_a,
   \qquad x\in\Omega.
\end{equation}
Since $\partial_jP_\Omega=S n_j$, there are $b\in\R^d$ and $C\in\R$ such that
\begin{equation}\label{eq:P-isotropic}
   P_\Omega(x)=-\frac{|x|^2}{2d}+b\cdot x+C,
   \qquad x\in\Omega.
\end{equation}
The equality statement in \cref{thm:endpoint-HS} implies that $\Omega$ is an ellipsoid. After a translation and an orthogonal change of coordinates, it has the form
$$
   \Omega=\left\{x\in\R^d:
       \sum_{i=1}^d\frac{x_i^2}{\rho_i^2}<1\right\},
   \qquad \rho_i>0.
$$
Suppose first that $d\ge3$. The interior Hessian of its Newtonian potential is diagonal with entries $-L_i$, where the depolarization factors are
\begin{equation}\label{eq:depolarization}
   L_i=\frac{\prod_{m=1}^d\rho_m}{2}
   \int_0^\infty
   \frac{\dd s}{(\rho_i^2+s)\prod_{m=1}^d(\rho_m^2+s)^{1/2}},
   \qquad
   \sum_{i=1}^dL_i=1.
\end{equation}
See, for example, \cite[Theorem~2.1]{DiFratta2016}. Equation \eqref{eq:P-isotropic} gives $L_i=1/d$ for every $i$. If $\rho_i>\rho_j$, the integrand defining $L_i$ is strictly smaller than the integrand defining $L_j$ for every $s>0$, so $L_i<L_j$. Hence equality of all depolarization factors forces
$$
   \rho_1=\cdots=\rho_d.
$$
When $d=2$, direct calculation of the logarithmic potential gives
$$
   L_i=\frac{\rho_1\rho_2}{2}
   \int_0^\infty
   \frac{\dd s}{(\rho_i^2+s)
   \bigl((\rho_1^2+s)(\rho_2^2+s)\bigr)^{1/2}},
   \qquad i=1,2.
$$
The identity
$$
   \frac{\dd}{\dd s}
   \left(\frac{\rho_2^2+s}{\rho_1^2+s}\right)^{1/2}
   =\frac{\rho_1^2-\rho_2^2}
   {2(\rho_1^2+s)^{3/2}(\rho_2^2+s)^{1/2}}.
$$
therefore yields
$$
   L_1=\frac{\rho_2}{\rho_1+\rho_2},
   \qquad
   L_2=\frac{\rho_1}{\rho_1+\rho_2}.
$$
Equation \eqref{eq:P-isotropic} gives $L_1=L_2=1/2$, and hence $\rho_1=\rho_2$. Thus $\Omega$ is a ball in every dimension.
\end{proof}

Combining \cref{thm:kyfan,cor:max-bound,thm:rigidity} proves \cref{thm:main}.

\begin{remark}[Regularity and topology]\label{rem:regularity}
The $C^{1,\alpha}$ assumption ensures compactness of the Neumann--Poincar\'e operator and gives the classical energy-space realisation used above. Connectedness of $\Omega$ enters through Neumann uniqueness in \cref{lem:affine-field}. The boundary may have several connected components; in an equality case, the flux argument in the proof of \cref{thm:endpoint-HS} excludes every bounded component of the complement.
\end{remark}

\paragraph{\textbf{AI declaration.}}
The proof ingredients in this paper came about through interactions between the authors and ChatGPT 5.6. We subsequently developed, refined, and independently checked all resulting arguments. The authors take full responsibility for all content.

\end{document}